\documentclass[12pt]{amsart}
\usepackage{amsmath,amssymb,latexsym, amsthm, amscd, mathrsfs, stmaryrd}
\usepackage[all]{xy}

\theoremstyle{definition}
\newtheorem{rem}[subsubsection]{Remark}
\theoremstyle{plain}
\newtheorem{prop}[subsubsection]{Proposition}
\newtheorem{thm}[subsubsection]{Theorem}
\newtheorem{lem}[subsubsection]{Lemma}
\newtheorem{cor}[subsubsection]{Corollary}

\newcommand{\Hom}{\mathrm{Hom}}
\newcommand{\mbf}{\mathbf}
\newcommand{\mbb}{\mathbb}
\newcommand{\mrm}{\mathrm}
\newcommand{\A}{\mathcal A}
\newcommand{\C}{\mbb C}

\newcommand{\F}{\mathcal F}
\newcommand{\tF}{\tilde \F}

\newcommand{\GL}{\mrm{GL}}
\newcommand{\K}{\mbf K}

\newcommand{\q}{q^{ \tiny \frac{1}{2}}}
\newcommand{\ro}{\mbox{\small ro}}
\newcommand{\co}{\mbox{\small  co}}

\newcommand{\bin}[2]{{\renewcommand\arraystretch{0.5}
\left[\!\begin{array}{c}
{\scriptstyle #1}\\{\scriptstyle #2}\end{array}\!\right]}}

\title[Geometric quantum  $\GL(n)$]
{On quantum $\mbf{GL}(n)$}

\author{  Yiqiang Li}
\address{Department of Mathematics\\   State University of New York at Buffalo\\
244 Mathematics Building, Buffalo, NY 14260}
\email{yiqiang@buffalo.edu}

\date{\today}

\begin{document}

\begin{abstract}
The quantum $\GL(n)$ of Faddeev-Reshetikhin-Takhtajan and Dipper-Donkin are realized geometrically by using double partial flag varieties. 
As a consequence, the difference of  these two Hopf algebras is caused by  a twist of a cocycle in the multiplication.  
\end{abstract}

\maketitle

\section{Introduction}

Let $\GL(n)$ be the general linear group and $\mathfrak{gl}(n)$  its Lie algebra. 
In ~\cite{BLM90},  a geometric realization of quantum $\mathfrak{gl}(n)$ (or rather the q-Schur algebra) is given by using double flag varieties. 
From such a construction arises the so-called modified quantum $\mathfrak{gl}(n)$ and its canonical basis. Such classes of algebras play important roles in the higher representation theory (\cite{KL10}). 

It is natural to ask if the quantum $\GL(n)$ itself admits a geometric realization. Since quantum $\GL(n)$ is, in principle,   dual to quantum $\mathfrak {gl}(n)$,  one may expect to get an answer from the dual construction of ~\cite{BLM90}.  
However, the answer to such a question is subtle because the group $\GL(n)$  admits  several quantizations: one by Faddeev-Reshetikhin-Takhtajan (\cite{FRT88}, ~\cite{FT86}), one by Dipper-Donkin (\cite{DD91}), 
 one by Takeuchi (\cite{T90}) and  one by Artin-Schelter-Tate (\cite{AST91}).

In this paper, we show that the dual construction of ~\cite{BLM90} together with the coproduct defined by Grojnowski (\cite{Grojnowski}) and Lusztig (\cite{Lu99}) is isomorphic to the quantum $\GL(n)$ of Dipper-Donkin. 
The quantum $\GL(n)$ of Faddeev-Reshetikhin-Takhtajan  is also obtained from this setting by twisting a cocycle on the multiplication. 
In the geometric realization of both quantizations, the comultiplication are the same. This shows that the two quantizations are isomorphic as coalgebras, which was  proved by Du, Parshall and Wang (\cite{DPW91}) by a different method twenty years ago. 
A closer look at the geometric construction yields that the basis $E^M$ in the quantum $\GL(n)$ is the same as the basis consisting of all characteristic functions of certain orbits up to a twist. 
One interesting fact in the geometric realization is that the quantum determinants in both quantizations get identified with a certain Young symmetrizer.  This symmetrizer gives rise to the determinant representation of $\GL(n)$ and    the parameter in the quantization process does not appear. 

The geometric setting  is suitable to investigate many topics on quantum $\GL(n)$ such as quantum Howe $\GL(m)$-$\GL(n)$ duality (see ~\cite{W01}, ~\cite{Zh03}) and  (dual) canonical basis of quantum $\GL(n)$   in ~\cite{Zhang04} and ~\cite[29.5]{Lusztig93}. 
We hope to come back to these topics in the future. 

\setcounter{tocdepth}{1}
\tableofcontents

\section{Quantum $\GL(n)$, definitions}

We refer to ~\cite{PW91}, ~\cite{T89}, ~\cite{KS97}, ~\cite{KS98}, ~\cite{M88}, ~\cite{NYM93},  ~\cite{W87a} and ~\cite{W87b} for more details. We shall recall the definitions of quantum $\GL(n)$.

\subsection{Quantum group $\GL_v(n)$  of Faddeev-Reshetikhin-Takhtajan}
\label{qmatrix}

Let $\C$ be the field of complex numbers. We fix a nonzero element $v$ in $\C$.  
Note that the parameter $v$ is denoted by $q$ in the literature, such as ~\cite{FRT88} and ~\cite{DD91}. We reserve the letter $q$ for the finite field $\mbb F_q$ of $q$ elements. 

Following ~\cite{FRT88}, the quantum matrix algebra $A_v(n)$ of Faddeev-Reshetikhin-Takhtajan is defined to be  the  unital associative algebra over $\C$ generated by the symbols $E_{st}$ for $1\leq i, j\leq n$ and subject to the following defining relations:
\begin{align*}
&E_{ik}E_{jl} =E_{jl}E_{ik}, & & \forall i>j, k<l;\\
&E_{ik}E_{jl}=E_{jl}E_{ik} + (v-v^{-1}) E_{jk}E_{il}, &&\forall  i>j, k>l;\\
& E_{ik}E_{il}=v E_{il}E_{ik}, && \forall k>l;\\
&E_{ik}E_{jk} =vE_{jk}E_{ik}, &&\forall i>j.
\end{align*}
The algebra $A_v(n)$ admits a bialgebra structure whose   comultiplication   
\[
\Delta: A_v(n) \to A_v(n)\otimes A_v(n)
\]
is defined by
\begin{align}
\label{A-co}
\Delta (E_{st}) =\sum_{k=1}^n E_{ik}\otimes E_{kj} \quad \forall 1\leq i, j\leq n,
\end{align}
and the counit
$\epsilon : A_v(n) \to \C$ is  given by $\epsilon (E_{st}) =\delta_{st}$ for any $1\leq i, j\leq n$.
Let 
\begin{align}
\label{FRT-determinant}
\mrm{det}_v = \sum_{\sigma\in S_n} (-v)^{-l(\sigma)} E_{1, \sigma(1)}\cdots E_{n,\sigma(n)},
\end{align}
be the quantum determinant of $A_v(n)$, where $S_n$ is the symmetric group of $n$ letters and $l$ the length function. 
It is well-known that $\mrm{det}_v$ is a central element in $A_v(n)$.
The quantum group  $\GL_v(n)$ of Faddeev-Reshetikhin-Takhtajan is (the would-be group whose coordinate ring is)  the algebra obtained by localizing $A_v(n)$ at $\mrm{det}_v$, i.e., 
\[
A_v(n) \otimes_{\C} \C[T]/ < T\mrm{det}_v -1>.
\]
The comultiplication $\Delta$ of $A_v(n)$ extends naturally to a comultiplication of $\GL_v(n)$, still denoted by $\Delta$,  of  $\GL_v(n)$. In particular, $\Delta (T) = T\otimes T$.
Since $T$ is the inverse of $\mrm{det}_v$, we have $T=\mrm{det}^{-1}_v$.

Fix two integers $i$ and $j$ among the set $\{1,\cdots, n\}$, consider the subalgebra of $A_v(n)$ generated by the generators $E_{k, l}$ for $k\neq i$ and $l \neq j$. 
The resulting algebra is isomorphic to $A_v(n-1)$. So its quantum determinant, denoted by $A(i, j)$,  is well defined. 
The antipode $S$ of $\GL_v(n)$ is defined by 
\begin{align}
\label{antipode}
S(E_{st}) = (-v)^{j-i} A(j, i) \mrm{det}_v^{-1}, \quad \forall 1\leq i, j \leq n.
\end{align}

The datum $ (\GL_v(n), \Delta, \epsilon, S)$ is a non-commutative, non-cocommutative Hopf algebra.

\subsection{Quantum group $\GL^{DD}_v(n)$ of Dipper-Donkin} 
\label{qmatrixDD}
Following ~\cite{DD91}, 
the quantum matrix algebra $B_v(n)$ is defined to be the unital associative algebra over $\C$ generated by the symbols $c_{st}$, for $1\leq i, j\leq n$, and subject to the following 
relations:
\begin{align*}
& c_{ik}c_{jl} = v c_{jl} c_{ik}, && \forall i> j, k\leq l,\\
& c_{ik}c_{jl} =c_{jl} c_{ik} + (v-1) c_{jk} c_{il}, && \forall i>j ,  k> l,\\
& c_{ik} c_{il} = c_{il} c_{ik}, && \forall i, l, k. 
\end{align*}
The triple  $(B_v(n), \Delta, \epsilon)$, where $\Delta $ and $\epsilon$ are defined   in Section ~\ref{qmatrix}, is again a bialgebra. 
Following ~\cite[4.1.7]{DD91}, let 
\begin{align}
\label{DD-determinant}
\mrm{det}^{DD}_v= \sum_{\sigma\in S_n} (-v)^{-l(\sigma)}  c_{\sigma(1), 1} \cdots c_{\sigma(n),n},
\end{align}
be the quantum determinant of $B_v(n)$.  (Note that this definition is equivalent to the other definitions in ~\cite{DD91} for $v$ invertible.) The quantum group $\GL^{DD}_v(n)$ is (the group whose coordinate algebra is ) the algebra obtained from $B_v(n)$  by localizing at $\mrm{det}^{DD}_v$. 

Similar to the definition of $A(i, j)$, we can define the element $A(i, j)^{DD}$. Then the antipode of $\GL_v^{DD}(n)$ is given by
\begin{align}\label{antipode-DD}
S^{DD}(c_{st}) = (-1)^{i+j} A^{DD}(j,i) \mrm{det}^{DD, -1}_v, \quad \forall 1\leq i, j\leq n,
\end{align}
where $\mrm{det}^{DD, -1}_v$ is the inverse of $\mrm{det}^{DD}_v$.

The datum $(\GL^{DD}_v(n), \Delta, \epsilon, S^{DD})$ is a Hopf algebra, where $\epsilon$ is defined in the same way as that of $\GL_v(n)$.

\subsection{Notation}
If the parameter $v$ is known, we simply substitute it by the number, 
for example, we have $A_{\q}(n)$, $\GL_{\q}(n)$, $B_q(n)$ and $\GL^{DD}_q(n)$.

\section{Geometric set up}

\label{GeomK}

\subsection{Operations}
\label{operations}

Let $f: Y\to X$ be a map between two finite sets. For a given complex-valued  function $\phi$ over $X$, we define a function $f^* \phi$ on $Y$ by 
\[
f^*(\phi) (y) = \phi(f(y)), \quad \forall y\in Y.
\]
For a given function $\psi$ over $Y$, we define a function $f_!(\psi)$ over $X$ by
\[
f_!(\psi)(x) =\sum_{y\in f^{-1}(x) } \psi(y), \quad \forall x\in X. 
\]
If the numbers of elements in the fibers of $f$  are the same, we define
\[
f_{\flat} = \frac{f_!}{\# f^{-1}(x)},  
\]
where $x$ is any element in $X$.

\subsection{Partial flags}

Let $\mbf  P$ be the monoid $\mbb N^n$ of all $n$-tuples $\mbf d=(d_1, d_2, \cdots, d_n)$  of non negative integers. The set $\mbf P$ admits a partition
\[
\mbf P =\sqcup_{d\in \mbb N} \mbf P_d, \quad \mbf P_d=\{ \mbf d \in \mbf P| d_1 +d_2+\cdots+ d_n = d\}.
\] 
Let $\Theta$ be the set of non negative integer valued matrix of size  $n\times n$. The set $\Theta$ admits  a partition
\[
\Theta =\sqcup_{d\in \mbb N} \Theta_d , \quad \Theta_d =\{ M=(m_{st}) \in \Theta | \sum_{1\leq i, j \leq n} m_{st} =d\}.
\]
We define two maps $\ro, \co : \Theta \to \mbf P$ by   
\[
\ro(M) = \left ( \sum_{j=1}^n m_{1j}, \cdots, \sum_{j=1}^n m_{nj}  \right ) \quad \mbox{and} \quad
\co(M) = \left (\sum_{i=1}^n m_{i1}, \cdots, \sum_{i=1}^n m_{in} \right ),
\]
for any $M\in \Theta$. We set $\Theta_d(\mbf c, \mbf d)$ to be the subset in $\Theta_d$ such that $\ro(M) =\mbf c$ and $\co(M) =\mbf d$. 

Let us fix a   finite  field  $\mbb F_q$ of  $q$ elements. 
An $n$-step flag $F=(0=F_0\subseteq F_1\subseteq \cdots \subseteq  F_n=\mbb F_q^d)$ is called  of type $\mbf d\in \mbf P_d$  if 
$\dim F_i/F_{i-1}= d_i$ for any $i=1, \cdots, n$.
Let $\F_{\mbf d}$ be the set of all partial flags of type $\mbf d$. When the vector space $\mbb F_q^d$ is replaced by a $d$ dimensional vector space, say $D$, we write 
$\F_{\mbf d}(D)$ for the set of all $n$-step partial flags in $D$ of type $\mbf d$.  We set 
\[
\F = \sqcup_{\mbf d\in \mbf P_d} \F_{\mbf d}.
\]
When we want to emphasis the dimension $d$, we write $\F^d$ for $\F$.

Let $\GL(d) = \mrm{GL}(\mbb F_q^d)$ be the general linear group of rank $d$. It is clear that $\GL(d)$ acts transitively from the left on $\F_{\mbf d}$ for any $\mbf d\in \mbf P_d$. 
This action induces an action on the double  partial flag varieties $\F_{\mbf c}\times \F_{\mbf d} $ for any $(\mbf c,\mbf d)\in \mbf P_d\times \mbf P_d$ by 
acting diagonally. Let $\GL(d) \backslash \F_{\mbf c} \times \F_{\mbf d}$ be the set of $\GL(d)$-orbits. We have the bstection
\[
\GL(d) \backslash \F_{\mbf c} \times \F_{\mbf d} \to \Theta_d(\mbf c, \mbf d) 
\]
 given by 
\[
(V, F) \mapsto M=(m_{st}), \quad \mbox{where} \quad m_{st}= \dim \frac{V_{i-1}+V_i\cap F_j}{V_{i-1}+V_i\cap F_{j-1}}, \forall 1\leq i, j \leq n
\]
for any $(V, F)\in \F_{\mbf c}\times \F_{\mbf d}$. We shall denote by $O_M$ the $\GL(d)$ orbit corresponding to the matrix $M\in \Theta_d$ under this bstection.

\subsection{Lusztig's diagram} 

For a pair $(V, F) \in \F_{\mbf c}\times \F_{\mbf d}$ and a subspace $U$ in $D=\mbb F_q^d$, we write
$(V, F)\cap U$ for the resulting pair of flags  in $\F (U)\times \F(U) $ obtained after intersecting each step with $U$. 
We also write $(V, F)\cap \frac{D}{U}$ for the pair of flags obtained by passage to the quotient $D/U$.

Let $\mrm G(d'', d)$ be the Grassmannian of $d''$ dimensional subspace in a fixed $d$ dimensional vector space over $\mbb F_q$.

For triple pairs $\mbf c'', \mbf d'' \in \mbf P_{d''}$, $\mbf c', \mbf d' \in \mbf P_{d'}$, and   $ \mbf c, \mbf d \in \mbf P_{d}$ such that $\mbf c''+\mbf c'=\mbf c$ and $\mbf d''+\mbf d'=\mbf d$, we
consider the following diagram
\begin{equation}
\label{induction}
\begin{CD}
(\F_{\mbf c''}\times \F_{\mbf d''}) \times (\F_{\mbf c'}\times \F_{\mbf d'} )  @<\pi_1 <<  \mathcal E' @>\pi_2>> \mathcal E'' @>\pi_3>> \F_{\mbf c} \times \F_{\mbf d},
\end{CD}
\end{equation}
 where
 \begin{align*}
& \mathcal E'' = \{ (V, F; E)\in \F_{\mbf c} \times \F_{\mbf d} \times \mrm G(d'', d)  | (V, F) \cap E \in    (\F_{\mbf c''}\times \F_{\mbf d''})\},\\
&\mathcal E' =\{ (V, F; E; R', R'')| (V, F; E) \in \mathcal E'', R': \mbb F_q^d/ E \tilde{\to} \mbb F_q^{d'}, R'': E\tilde\to \mbb F_q^{d''}\},
 \end{align*}
 and the morphisms $\pi_2$ and $\pi_3$ are projections and 
 \[
 \pi_1(V, F; E; R', R'') = (R''((V, F) \cap E), R'( (V, F)\cap \frac{\mbb F_q^d}{E}) ).
 \]
 
 We emphasis that the left end of the above diagram is in opposite direction to that of the usual Lusztig's diagram.
 
  Note that, if defined over an algebraic closure of $\mbb F_q$,  $\pi_1$ is a smooth morphism of fiber  dimension 
 \begin{equation}
 \label{pi1}
 f_1= \sum_{i<j} (c_i'c_j'' + d_i' d_j'') + d'd'' +d'd' + d'' d''.
 \end{equation}
 The morphism $\pi_2$ is a $\GL(d')\times \GL(d'')$ principal bundle, hence the fiber dimension is 
 \begin{equation}
 \label{pi2}
 f_2=d'd' +d'' d''.
 \end{equation}
 The morphism $\pi_3$ is a compactifiable morphism, though not proper in general.  Recall that a compactifiable morphism $f$ is a morphism that can be written as $f_1f_2$ where $f_2$ is an open embedding while $f_1$ is proper.

\section{Bialgebra $(\K_{\q}(n), \circ, \tilde \Delta)$}

\subsection{Multiplication}
\label{mult}

Let 
\[
\K_{\q}(n) =\oplus_{d\in \mbb N} \K_{\q,d}(n) , \quad 
\K_{\q, d}(n) =\mrm{span}_{\C} \{ 1_{M}| M\in \Theta_d\},
\]
where $1_M$ is the characteristic function of the orbit $O_M$.
We define 
\begin{equation}
\label{twist}
1_{M''}\circ 1_{M'} = (q^{-\frac{1}{2}})^{f_1-f_2} \pi_{3!} \pi_{2\flat} \pi_1^* (1_{M''}\otimes 1_{M'}),
\end{equation}
where $f_1$ and $f_2$ are from (\ref{pi1}) and (\ref{pi2}), respectively,
and the notations $\pi_1^*$, etc are given in Section ~\ref{operations}.
Note that 
\[
f_1-f_2=\sum_{i<j} (c_i'c_j'' + d_i' d_j'') + d'd'' ,
\]
where $(\ro (M''), \co (M'') )= (\mbf c'', \mbf d'')$ and $(\ro(M'), \co (M')) =(\mbf c', \mbf d')$. 
Note that the operation $\circ$ can be  extended  to a linear map
\[
\circ: \K_{\q, d'}(n) \otimes \K_{\q, d''}(n) \to \K_{\q, d'+d''}(n).
\]

\begin{prop}
\label{associative}
The multiplication $``\circ"$ is associative.
\end{prop}

\begin{proof}
It suffices to show that 
\[
(1_{M'''}\circ 1_{M''}) \circ 1_{M'} = 1_{M'''} \circ (1_{M''} \circ 1_{M'}),
\]
for a triple of matrices ($M', M'', M''')$ whose corresponding orbits are in $\F_{\mbf c'}\times \F_{ \mbf d'}$, $\F_{\mbf c''}\times \F_{\mbf d''}$ and $\F_{\mbf c'''}\times \F_{\mbf d'''}$, respectively. It is straightforward to  check that the shifts on the left hand side and the right hand side are equal to 
\[
\sum_{i<j} c_i'c_j''+c_i'c_j'''+c_i''c_j'''+d_i'd_j'' +d_i'd_j''' +d_i'' d_j'''+ d'd''+d'd'''+d''d'''.
\]
Moreover, without twisted, we have the evaluations of 
$(1_{M'''} \circ 1_{M''}) \circ  1_{M'} $ and $1_{M'''} \circ   (1_{M''} \circ 1_{M'})$ at $(V, F)\in \F_{\mbf c} \times \F_{\mbf d}$
is equal to 
\begin{align*}
\#\{ E_1\subset E_2 \subset D | (V, F)\cap E_1 \in O_{M'''},
(V, F) \cap \frac{E_2}{E_1}\in O_{M''}, (V, F) \cap \frac{D}{E_2} \in O_{M'}\},
\end{align*}
where $D$ is a vector space over $\mbb F_q$ of dimension $d'+d''+d'''$.
From the above analysis, we see that the proposition follows.
\end{proof}

From the above proposition, we see that  the pair $(\K_{\q}(n), \circ)$ is an associative algebra. Note that the unique element, the zero matrix, in $\Theta_0$ is the unit of the algebra $\K_{\q}(n)$.

\subsection{Defining relations}
For each pair $(i, j)$, let $e_{st}$ be the element in $\Theta_1$ whose value at the entry $(i, j)$ is $1$ and zero elsewhere. 
It is clear that the corresponding orbit consists of  only a single point
$\F_{e_i}\times \F_{e_j}$ where the set $\{e_i| i=1,\cdots, n\} $ is the standard basis of $\mbb Z^n$.

Let $E_{st}$ be the characteristic function on the orbit $O_{e_{st}}=\F_{e_i}\times \F_{e_j}$.

We also set
\[
E_{st}^{(n)}=\frac{E_{st}^n}{[n]^!_{\q}}, \quad \forall n\in \mbb N.
\]
where
\[
[n]^!_{\q} = [1]_{\q} [2]_{\q} \cdots [n]_{\q}, \quad [m]_{\q} = \frac{(\q)^{m} -(\q)^{-m}}{\q -q^{-\frac{1}{2}}}.
\]
We also set
\[
\bin{m}{n}_{\q} = \frac{[m]^!_{\q}}{[n]^!_{\q} [m-n]_{\q}^!},\quad \forall m \geq  n \in \mbb N.
\]
\begin{lem}
\label{relation-1}
$E_{st}^{(n)} = 1_{n e_{st}}$, $E_{st} \circ E_{st}^{(n)} =[n+1]_{\q} E_{st}^{(n+1)}$ and 
$E_{st}^{(m)} \circ  E_{st}^{(n)} =\bin{m+n}{m}_{\q} E_{st}^{(m+n)}$. 
\end{lem}

This lemma is due to the fact that $1_{e_{st}} 1_{ne_{st}}=[n+1]_{\q} 1_{(n+1)e_{st}}$, which can be calculated by definition.

\begin{lem}
\label{relation-a}
$E_{ik} \circ E_{jl}= E_{jl} \circ E_{ik} $, for any $i>j$, $k<l$.
\end{lem}

\begin{proof}
The corresponding pair $(\mbf {c, d}) $ in this situation is $( e_i+e_j, e_k+e_l)$. 
The set $\F_{\mbf c}\times \F_{\mbf d}$ has two orbits corresponding to the matrices
$e_{jk}+e_{il}$ and $e_{jl}+e_{ik}$.
By definition, we have, for any $(V, F)\in \F_{\mbf c}\times \F_{\mbf d}$,
\begin{equation}
\label{formalism}
 E_{jl}E_{ik} = (q^{-\frac{1}{2}})^{f_1-f_2} \#X, \quad
 E_{ik}E_{jl} =(q^{-\frac{1}{2}})^{f_1-f_2} \#Y,
\end{equation}
where 
\begin{align*}
X &= \{ E\in \mrm G(1, 2) |  ( V, F) \cap E\in O_{e_{jl}}, (V, F) \cap \frac{\mbb F_q^2}{E}  \in O_{e_{ik}}\},
\\
Y &=\{ E\in \mrm G(1, 2) |  ( V, F) \cap E \in O_{e_{ik}}, (V, F) \cap \frac{\mbb F_q^2}{E}  \in O_{e_{jl}}\}.
\end{align*}

In both cases, $f_1+f_2=2$.  
If  $(V, F)\in O_{e_{jk}+e_{il}}$, we have $V_j=F_k$. This implies that $X=Y=$\O.
If $(V, F)\in O_{e_{jl}+e_{ik}}$, we have $V_j\neq F_k$.  This implies that 
$X=\{E| E=V_j\} $ and $Y=\{E| E=F_k\}$. So we have $\# X = \# Y =1$.  Therefore, the lemma holds.
\end{proof}

\begin{lem}
$E_{ik}\circ E_{jl}=E_{jl}\circ E_{ik} + (\q- q^{-\frac{1}{2}}) E_{jk}\circ E_{il}$, for any  $ i>j$, $k>l$.
\end{lem}

\begin{proof}
In this case, the pair $(\mbf { c, d})$ is still $(e_i+e_j, e_k+e_l)$. The following diagram is either the shift $f_1-f_2$ or the value of the function in the column at the given point in each row:

\begin{center}
\begin{tabular}{|l |l|c|c|c|c|} \hline
 {\bf }    & {\bf $E_{jl} E_{ik} $ }  & {\bf $E_{ik}E_{jl}$ } & $E_{jk} E_{il}$   \\ \hline \hline

$f_1-f_2$ & 1 &3  &2\\ \hline\hline
$(V, F)\in O_{e_{jl}+ e_{ik}}$ &1 &$q$ &0\\ \hline\hline

$(V, F)\in O_{e_{jk}+e_{il}}$ &0  & $q-1$ &1\\ \hline
                              
\end{tabular}
\end{center}
From the data in the above diagram, we see that the lemma follows.
\end{proof}

\begin{lem}
$E_{ik}\circ E_{il}=\q E_{il} \circ E_{ik}$, for any  $k>l$.
\end{lem}

\begin{proof}
The pair $(\mbf {c, d})$ is $(2e_i, e_k+e_l)$ in this case. In this case, the set $\F_{2e_i}\times \F_{e_k+e_l}$ has only one $\GL(2)$-orbit corresponding to $e_{ik}+e_{il}$.   We have the following data:
\begin{center}
\begin{tabular}{|l |l|c|c|c|} \hline

 {\bf }  & {\bf $E_{il} E_{ik} $ } 
 & {\bf $E_{ik}E_{il}$ }     \\ \hline \hline

$f_1-f_2$ & 1 &2 \\ \hline\hline
$(V, F)\in O_{e_{ik}+e_{il}}$ &1 &$q$\\ \hline

\end{tabular}
\end{center}
The lemma follows from the above data.
\end{proof}

\begin{lem}
\label{relation-d}
$E_{ik}\circ E_{jk} =\q E_{jk}\circ E_{ik}$, for any $i>j$.
\end{lem}

\begin{proof}
The pair $(\mbf {c, d})$ in this case is $(e_i+e_j, 2e_k)$. The set $\F_{e_i+e_j}\times\F_{ 2e_k}$ 
has only one $\GL(2)$-orbit corresponding to $e_{ik}+e_{jk}$. 
\begin{center}
\begin{tabular}{|l |l|c|c|c|} \hline

 {\bf }  & {\bf $E_{jk} E_{ik} $ }   & {\bf $E_{ik}E_{jk}$ }    \\ \hline \hline

$f_1-f_2$ & 1 &2 \\ \hline\hline
$(V, F)\in O_{e_{ik}+e_{jk}}$ &1 &$q$\\ \hline

\end{tabular}
\end{center}
The lemma follows from the above diagram.
\end{proof}

From Lemmas \ref{relation-a}-\ref{relation-d}, we see that 

\begin{prop}\label{defining}
Under the multiplication ``$\circ$'', 
the functions $E_{ij}$ in $\K_{\q}(n)$, for any $1\leq i, j\leq n$, satisfy the defining relations for 
the quantum matrix algebra $A_v(n)$ in Section ~\ref{qmatrix} for  $v=\q$.
\end{prop}

\subsection{Generators}

We put
\begin{align}
\label{dM}
d(M) =\sum_{i, j, k, l} m_{st} m_{kl},
\end{align}
where $i$, $j$, $k$ and $l$ run from $1$ to $n$ and satisfy that either  $i<k$ or $j<l$. From ~\cite[ 2.2(b)] {BLM90},  $d(M)$ is the dimension of $O_M$ (if defined over an algebraically closed field).
We write
\[
E^{(M)} =\prod_{1\leq i, j\leq n} E_{st}^{(m_{st})},
\]
where the product is taken in the   lexicographic order ``$\leq$''  on the set $\{(i, j)| 1\leq i, j\leq n\}$. For example, if $n=2$, we have 
\[
E^{(M)} = E_{11}^{(m_{11})}E_{12}^{(m_{12})} E_{21}^{(m_{21})} E_{22}^{(m_{22})}  .
\]

\begin{prop}
\label{basis}
We have 
\begin{align}
\label{PBW}
E^{(M)} = (q^{-\frac{1}{2}})^{d(M)} 1_M ,  \quad \forall M\in \Theta_d.
\end{align}
In particular, the set $\{ E_{st}| 1\leq i, j\leq n\}$ generates the algebra $(\K_{\q}(n)$, $\circ)$. 
\end{prop}

The proof of Proposition ~\ref{basis} will be given after the proof of Lemma ~\ref{technical}.

 An $n^2$-step partial flag 
\[
D_{\bullet}=\{ 0=D_{10} \subseteq D_{11} \subseteq \cdots \subseteq D_{1n} \subseteq D_{21}\subseteq \cdots \subseteq D_{2n}\subseteq \cdots \subseteq D_{nn}=D) \]
is called {\em of type $M\in \Theta_d$}  if  
\[
\dim D_{st}/D_{i,j-1} = m_{st}, \quad \forall 1\leq  i, j \leq n,
\]
where we set $D_{i0}=D_{i-1, n}$. Let $\F_M$ be the set of all  flags of type $M$.  
Given any pair $(V, F) \in \F_{\mbf c}\times \F_{\mbf d}$, where $O_M$ lies in,  we can associate an $n^2$-step flag by 
defining $D_{st} =V_{i-1} +V_i\cap F_j$ for any $1\leq i, j\leq n$. 
We say that
$(V, F) $ {\em is of type $M$} if the associated $n^2$-step flag is of type $M$.

 Let $\tF_M$ be the set of triples $(V, F, D_{\bullet})$, where $(V, F)\in \F_{\mbf c}\times \F_{\mbf d}$ and $D_{\bullet}\in \F_M$,  such that 
the pair   $(V, F) \cap \frac{D_{st}}{D_{i, j-1}}$ is of type $m_{st} e_{st}$, for any $1\leq i, j\leq n$. We may organize the above-mentioned set in the following diagram:
\[
\begin{CD}
\tF_M @>\pi_{  M} >> \F_{\mbf c}\times \F_{\mbf d} \\
@VpVV @.\\
\F_M
\end{CD}
\]
where the maps $\pi_M$ and $p$ are natural projections. 

\begin{lem}
\label{technical}
The morphism  $\pi_M$ is an injective map whose  image is  $O_M$. In other words, $\tF_M$ is isomorphic to $O_M$. 
Moreover,  the fibers of $p$ are   vector spaces of dimension $ \sum_{i>k, j<l} m_{st} m_{kl}$.
\end{lem}

\begin{proof}

Suppose that $(V, F, D_{\bullet})$ is a triple in $\tF_M$. Then the fact that 
$(V, F) \cap \frac{D_{1j}}{D_{1, j-1}}$ is of type $m_{1j}e_{1j}$, for $1\leq j\leq n$,  implies that 
\[
\dim \frac{V_1\cap F_j}{D_{1, j-1} } =m_{1j}. 
\]
Hence, we have 
\[
V_1=D_{1n}.
\]
Inductively, we can prove that 
\[
V_i = D_{in}, \quad  \forall 1\leq i\leq n.
\]
The condition that $(V, F)\cap \frac{D_{st}}{D_{i, j-1}}$ is of type $m_{st}e_{st}$ implies that 
\begin{align}
\label{tech-1}
\dim \frac{F_j\cap D_{st}}{D_{i, j-1}} = m_{st}, \quad 
\dim \frac{F_{j-1}\cap D_{st} } {D_{i, j-1}} =0, \quad  \forall 1\leq i, j\leq n.
\end{align}
Fix  $D_{\bullet}$ (hence $V$), we see  from (\ref{tech-1}) that the choice of $F$ such that $(V, F, D_{\bullet})\in \tF_M$ is isomorphic to the vector space 
\[
T=\oplus_{j=1}^{n-1} \Hom (E_{st}, \oplus_{i>k, j<l} E_{kl}), \quad \mbox{where} \quad  E_{st} =D_{st}/D_{i, j-1}.
\]
From this, we see that $p$ is a vector bundle of the above-mentioned fiber dimension. 
Moreover, given any element $\phi=(\phi_j) \in T$, the corresponding flag $F=(F_j)$ satisfies that 
$F_j$ is the linear space spanned by $F_{j-1}$ and the elements $v_{st}+ \phi_j(v_{st})$ for any $v_{st}\in E_{st}$ for $i=1,\cdots, n$. 
In particular, we see that $E_{st} \subseteq D_{i-1, n} + D_{in}\cap F_j$. 
So we have 
\[
V_{i-1} + V_i\cap F_j = D_{i-1, n} + D_{i n} \cap F_j = D_{st}. 
\]
This shows that  the associated $n^2$-step flag for the pair $(V, F)$ in the triple $(V, F, D_{\bullet})$ is $D_{\bullet}$.
The lemma follows.
\end{proof}

We are ready  to prove Proposition ~\ref{basis}.
In the definition of multiplication `$\circ$' in (\ref{twist}), we have a twist $(q^{-\frac{1}{2}})^{f_1-f_2}$ where $f_1-f_2= \sum_{i<j} (c_i'c_j''+d_i'd_j'' )+d'd''$.  
The element $E^{(M)}$ is obtained by carrying out $(n^2-1)$ times of multiplications. 
Altogether, the term $\sum_{i<j} c_i'c_j''$ contributes nothing, while the terms $\sum_{i<j} d_i'd_j'' $ and $d'd''$ contribute 
$\sum_{i>k, j<l} m_{st} m_{kl}$ and $\sum_{(i,j)>(k,l)}m_{st} m_{kl} =\sum_{(i,j)<(k,l)} m_{st}m_{kl}$, respectively, where $(i,j) <(k, l) $ is the lexicographic order.
By (\ref{dM}) 
the twists $(q^{-\frac{1}{2}})^{f_1-f_2}$ is $(q^{-\frac{1}{2}})^{d(M)}$ in $E^{(M)}$. 

By the definition of the multiplication, we have 
\begin{align}
E^{(M)} = (q^{-\frac{1}{2}})^{d(M)} \pi_{M!} (1_{\tF_M}) =(q^{-\frac{1}{2}})^{d(M)} 1_M, 
\end{align}
where the second equality is due to  Lemma ~\ref{technical}. We see that (\ref{PBW}) follows.
This finishes the proof of Proposition ~\ref{basis}.

Note that the symbol $E^{(M)}$ is also meaningful in the algebra $A_{\q}(n)$ and, moreover, they form a basis for $A_{\q}(n)$ (\cite[Theorem 1.4]{NYM93}). 
By taking account of this fact,  Propositions ~\ref{associative}, ~\ref{defining} and ~\ref{basis}, we have the following  theorem.

\begin{thm}
\label{Phi}
The assignment of sending    $E_{st}$ in the quantized matrix algebra $A_{\q}(n)$ of 
Faddeev-Reshetikhin-Takhtajan defined in Section ~\ref{qmatrix} to the element in the  same notation  in $\K_{\q}(n)$ defines 
an isomorphism of associative  algebras:
\[
\Phi_q: A_{\q}(n) \to (\K_{\q}(n), \circ) .
\]
\end{thm}

\subsection{Algebra homomorphism $\tilde \Delta$}
Given any triple $L, M, N\in \Theta_d$, we set 
\[
c^L_{M, N} =\# \{ \tilde F\in \F | (V, \tilde F)\in O_M, (\tilde F, F)\in O_N\},
\]
where $(V, F)$ is a fixed element in $O_L$. It is clear that $c^L_{M, N}$ is independent of the choice of the pair $(V, F)$.
We set 
\[
a_L = \# \mrm{Stab}_{\GL(d)} (V, F),
\] 
the stabilizer of $(V, F)$ in $\GL(d)$,  where $(V, F)$ is a fixed element in $O_L$.
We define a linear map 
\begin{equation}
\label{D}
\tilde \Delta_d : \K_{\q, d}(n) \to \K_{\q, d}(n)\otimes \K_{\q, d} (n),
\end{equation}
by 
\[
\tilde \Delta_d ( 1_L ) = \sum_{M, N\in \Theta_d} (q^{-\frac{1}{2}} )^{\frac{3}{2}d^2} 
\frac{a_M a_N}{a_L} c^L_{M, N} 1_M\otimes 1_N, \quad \forall L\in \Theta_d.
\]

By summing up all the linear maps $\Delta_d$, we have a linear map 
\begin{align}
\label{Delta}
\tilde \Delta: \K_{\q}(n)\to \K_{\q}(n) \otimes \K_{\q}(n).
\end{align}

\begin{prop}
\label{Delta-1}
$\tilde \Delta $ is an algebra homomorphism with respect to $``\circ"$.
\end{prop}

The proof will be given in Section ~\ref{comparison}.

Note that the shift $(\q)^{\frac{3}{2}d^2}$ in (\ref{D})  is annoying, moreover  we don't have $\tilde \Delta (E_{st} ) =\sum_{k=1}^n E_{ik}\otimes E_{kj}$ for $1\leq i, j \leq n$.  However, we can remedy this defect by modifying the multiplication
``$\circ$'' and $\tilde \Delta$ as follows:
\begin{equation}
\label{circ'}
\begin{split}
1_{M''} \circ' 1_{M'} & = (q^{-\frac{1}{2}})^{-d'd''} 1_{M''}\circ 1_{M'};\\
\Delta' (1_{L}) &= (\q)^{\frac{3}{2}d^2}   \tilde \Delta(1_{L}).
\end{split}
\end{equation}
That is we get rid of the twist $d'd''$ in the multiplication ``$\circ$'' and $\frac{3}{2}d^2$ in the comultiplication ``$\tilde \Delta$''.  If we set
\[
E'_{st} = a_{e_{st}} 1_{e_{st}} =(q-1) 1_{e_{st}}, \quad \forall 1\leq i, j \leq n.
\]
Then the set $\{ E'_{st} | 1\leq i, j\leq n\}$ satisfies the defining relations of $A_{\q}(n)$ and moreover 
$\Delta' (E'_{st}) =\sum_{k=1}^n E'_{ik} \otimes E'_{kj}$, for any $1\leq i, j\leq n$.

Define  an algebra homomorphism $\epsilon: \K_{\q}(n) \to \C$ by $\epsilon (E'_{st} )= \delta_{st}$ for any $1\leq i, j \leq n$. 
By using Proposition ~\ref{Delta-1}, we can strengthen 
 Theorem ~\ref{Phi} as follows. 

\begin{thm}
\label{bialgebra}
The morphism $\Phi'_q$ defined by $\Phi'_q( E_{st}) =E_{st}'$ for any $1\leq i, j\leq n$  is an isomorphism of bialgebras from 
$A_{\q}(n)$ to  the  bialgebra  $(\K_{\q}(n), \circ', \Delta')$ is equipped with the multiplication $\circ'$ and  comultiplication $\Delta'$ in 
$(\ref{circ'})$ and counit $\epsilon$.
\end{thm}

\section{Bialgebra ($\K_{\q}(n), \cdot, \Delta$)}

\subsection{Bialgebra ($\K_{\q}(n), \cdot, \Delta$)}
 In this subsection, we present a more natural bialgebra structure on $\K_{\q}(n)$, isomorphic to $A_{\q}(n)$.

For a triple $(M, M'', M')\in \Theta_d\times \Theta_{d''} \times \Theta_{d'}$ 
such that  
\begin{align}
\label{row}
\ro(M) = \ro(M'')+\ro(M') 
\quad \mbox{and} \quad 
\co(M) =\co (M'') + \co (M'),
\end{align}
we fix a flag $V\in \F_{\ro(M)}$, an element  $(V'', F'')  \in O_{M''}$ and $(V', F')\in O_{M'}$ and a decomposition 
$D=D'' \oplus D'$ in corresponding to the identity $d=d''+d'$. Let
\[
h^M_{M'', M'} = \# \{ F\in \F_{\co (M)} | (V, F)\in O_M, (V, F)\cap D'' =(V'', F'') , (V, F)\cap \frac{D}{D''}=(V', F')\}. 
\]
If the condition (\ref{row}) fails, we set $h^M_{M'', M'}=0$.
Then the number $h^M_{M'', M'}$ is independent of the choices of $V$, $(V'', F'')$ and $(V', F')$, due to ~\cite[1.2]{Lu00}. 

We define a linear map 
\begin{equation}
\label{newmult}
\cdot: \K_{\q}(n)\otimes \K_{\q} (n) \to \K_{\q}(n),
\end{equation}
by 
\[
1_{M''} \cdot 1_{M'} =
(q^{  -\frac{1}{2}})^{\sum_{i<j} -c_i'c_j''+d_i'd_j''}  \sum_{M\in \Theta}   h^M_{M'', M'} 1_M, \quad \forall M'', M'\in \Theta,
\]
where we set $\mbf c'' = \ro(M'')$,  $\mbf d''=\co (M'')$, $\mbf c' = \ro(M'')$ and $\mbf d'=\co (M'')$. The linear map ``$\cdot$''  is associative.   Indeed, it is enough to show that 
\begin{equation}
\label{associative-1}
1_{M'''}\cdot  (1_{M''}\cdot 1_{M'})= (1_{M'''} \cdot 1_{M''} ) \cdot 1_{M'} 
\end{equation}
The left hand side is equal to 
\[
(q^{-\frac{1}{2}} )^{ \sum_{i<j} -(c_i'c_j'' + c_i' c_j'''+ c_i'' c_j''') + (d_i'd_j''+d_i'd_j'''+ d_i'' d_j''') } 
\sum_{N\in \Theta} \sum_{ M\in \Theta} h^M_{M'', M'} h^N_{M''', M} 1_N.
\]
The right hand side is equal to 
\[
(q^{-\frac{1}{2}} )^{ \sum_{i<j} -(c_i'c_j'' + c_i' c_j'''+ c_i'' c_j''') + (d_i'd_j''+d_i'd_j'''+ d_i'' d_j''') } 
\sum_{N\in \Theta}\sum_{\tilde M \in \Theta} h^{\tilde M}_{M''', M''} h^N_{\tilde M, M'} 1_N.
\]
To show (\ref{associative-1}), it boils down to show that 
\[
\sum_{ M\in \Theta} h^M_{M'', M'} h^N_{M''', M}=
\sum_{\tilde M \in \Theta} h^{\tilde M}_{M''', M''} h^N_{\tilde M, M'}.
\]
Both sides are equal to the  quantity $h^N_{M''', M'', M'}$ defined in a similar way as $h^M_{M'', M'}$.  Therefore the identity (\ref{associative-1}) follows. From this, we see that the pair $(\K_{\q}(n), \cdot)$ is an associative algebra. 
We set 
\[
E_{st} = 1_{e_{st}}, \quad \forall 1\leq i, j\leq n.
\]

\begin{prop} 
\label{newdefining}
We have 
\begin{align*}
& E_{st}^m= 1_{me_{st}}, && \forall 1\leq i, j \leq n; m\in \mbb N;\\
&E_{ik}\cdot E_{jl} =E_{jl}\cdot E_{ik}, & & \forall i>j, k<l;\\
&E_{ik}\cdot E_{jl}=E_{jl} \cdot E_{ik} + (\q-q^{-\frac{1}{2}}) E_{jk}\cdot E_{il}, &&\forall  i>j, k>l;\\
& E_{ik}\cdot E_{il}=\q  E_{il}\cdot E_{ik}, && \forall k>l;\\
&E_{ik}\cdot E_{jk} =\q E_{jk}\cdot E_{ik}, &&\forall i>j,
\end{align*}
where the multiplication is taken under $\cdot$ in (\ref{newmult}).
\end{prop}

\begin{proof}
The proof is very much similar to that of Lemmas \ref{relation-1}-\ref{relation-d}.  The first identity is due to the fact that 
$h^{(m''+m')e_{st}}_{m''e_{st}, m' e_{st}} =1$. For any $i>j$, $k<l$, we have 
\begin{center}
\begin{tabular}{|l |l|c|c|c|} \hline

 {\bf }  & {\bf $E_{ik} E_{jl} $ }   & {\bf $E_{jl}E_{ik}$ }    \\ \hline \hline

The shift  & -1 &1 \\ \hline\hline
$(V, F)\in O_{e_{jk}+e_{il}}$ &0 & 0\\ \hline
$(V, F) \in O_{e_{jl}+ e_{ik}}$ & 1 & $q$ \\ \hline 
\end{tabular}
\end{center}

\noindent
For any $i>j, k>l$, we have 

\begin{center}
\begin{tabular}{|l |l|c|c|c|c|} \hline
 {\bf }    & {\bf $E_{jl} E_{ik} $ }  & {\bf $E_{ik}E_{jl}$ } & $E_{jk} E_{il}$   \\ \hline \hline

The shift  & 0 &0  &1\\ \hline\hline
$(V, F)\in O_{e_{jl}+ e_{ik}}$ &1 &$1$ &0\\ \hline\hline

$(V, F)\in O_{e_{jk}+e_{il}}$ &0  & $q-1$ &$q$\\ \hline
                              
\end{tabular}
\end{center}

\noindent
For any $k>l$, we have 
\begin{center}
\begin{tabular}{|l |l|c|c|c|} \hline

 {\bf }  & {\bf $E_{il} E_{ik} $ } 
 & {\bf $E_{ik}E_{il}$ }     \\ \hline \hline

The shift  & 0 &1 \\ \hline\hline
$(V, F)\in O_{e_{ik}+e_{il}}$ &1 &$q$\\ \hline

\end{tabular}
\end{center}

\noindent
For any $i>j$, we have 
\begin{center}
\begin{tabular}{|l |l|c|c|c|} \hline

 {\bf }  & {\bf $E_{jk} E_{ik} $ }   & {\bf $E_{ik}E_{jk}$ }    \\ \hline \hline

The shift  & 0 &-1 \\ \hline\hline
$(V, F)\in O_{e_{ik}+e_{jk}}$ &1 &$1$\\ \hline

\end{tabular}
\end{center}
The rest of the identities follow from the above  computations.
\end{proof}

Similar to the element $E^{(M)}$, we set
\[
E^M =\prod_{1\leq i, j\leq n} E_{st}^{m_{st}},
\]
where the product is taken in the lexicographic order and under the multiplication ``$\cdot$''. 

\begin{prop}
\label{newbasis}
Under the multiplication ``$\cdot$'', we have 
\[
E^M = q^{\frac{1}{2}\sum_{i>k, j<l} m_{st} m_{kl}} 1_M. 
\]
As a consequence, the set $\{ E_{st}| 1\leq i, j\leq n\}$ generates the algebra $(\K_{\q}(n),\cdot )$.
\end{prop}

\begin{proof}
As in the Proof of  Proposition ~\ref{basis}, we see that the shift in taking the products in $E^M$  is 
$(q^{-\frac{1}{2}})^{\sum_{i>k, j<l} m_{st} m_{kl}}$.  Moreover, we have 
\[
E^M = (q^{-\frac{1}{2}})^{\sum_{i>k, j<l} m_{st} m_{kl}} \sum_{\tilde M\in \Theta} 
h^{\tilde M}_{m_{11}e_{11}, \cdots, m_{1n}e_{1n}, m_{12}e_{12},\cdots, m_{nn}e_{nn}} 1_{\tilde M},
\]
where the structure constant $h^{\tilde M}_{m_{11}e_{11}, \cdots, m_{1n}e_{1n}, m_{12}e_{12},\cdots, m_{nn}e_{nn}}$ is defined in a similar manner as $h^M_{M'', M'}$. The argument in the Proof of Lemma ~\ref{technical} shows that 
this structure constant  is equal to $q^{\sum_{i>k, j<l} m_{st} m_{kl}}$ when $\tilde M=M$ and zero otherwise. So we have the identity in the Proposition. We are through.
\end{proof}

We define a bilinear map 
\begin{equation}
\label{newcomult}
\Delta: \K_{\q}(n) \to \K_{\q}(n) \otimes \K_{\q}(n)
\end{equation}
by
\[
\Delta (1_L) = \sum_{M, N\in\Theta} c^{L}_{M, N} 1_M  \otimes 1_N, \quad \forall  L\in \Theta
\]
where 
\[
c^L_{M, N} =\# \{ \tilde F\in \F | (V, \tilde F) \in O_M, (\tilde F, F)\in O_N\},
\]
for a fixed element $(V, F) \in O_L$. It is clear that $c^L_{M, N}$ is independent of the choices of the pair $(V, F)$. 
Moreover, $c^L_{M, N}$ is the structure constants for the $q$-Schur algebra defined in ~\cite{BLM90}.

\begin{prop}
\label{newalgebrahomo}
We have 
\[
\Delta (E_{st}) =\sum_{k=1}^n E_{ik} \otimes E_{kj} , \quad \forall 1\leq i, j \leq n. 
\]
Moreover, $\Delta$ is an algebra homomorphism with respect to the map $``\cdot "$ in $(\ref{newmult})$.
\end{prop}

\begin{proof}
It is clear that the identity in the Proposition holds. We are left to show that $\Delta$ is an algebra homomorphism. This boils down to show that 
\[
\Delta(1_{L''}\cdot 1_{L'}) = \Delta(1_{L''}) \cdot \Delta (1_{L'}), \quad \forall L'', L'\in \Theta.
\]
The left hand side is equal to 
\[
(q^{-\frac{1}{2}})^{\sum_{i<j} -c_i'c_j'' +d_i'd_j''} \sum_{M, N} \sum_{L} h^L_{L'', L'} c^L_{M, N} 1_M\otimes 1_N,
\]
while the right hand side is 
\[
(q^{-\frac{1}{2}})^{\sum_{i<j} -c_i'c_j'' +d_i'd_j''} \sum_{M, N} 
\sum_{M'', N'', M', N'}
h^M_{M'', M'} h^N_{N'', N'} c^{L''}_{M'', N''} c^{L'}_{M', N'} 1_M\otimes 1_N.
\]
Thus, it reduces to show that 
\begin{equation}
\label{Green}
\sum_{L} h^L_{L'', L'} c^L_{M, N}=
\sum_{M'', N'', M', N'}
h^M_{M'', M'} h^N_{N'', N'} c^{L''}_{M'', N''} c^{L'}_{M', N'}.
\end{equation}
This is  proved in ~\cite[Proposition 1.5]{Lu00} and ~\cite{Grojnowski}.
\end{proof}

By Proposition ~\ref{newdefining}, ~\ref{newbasis} and ~\ref{newalgebrahomo}, we have 

\begin{thm}
\label{FRT-isom}
The map defined by sending the generators, $E_{st}$,  in $A_{\q}(n)$ to the elements with the same notation in $\K_{\q}(n)$ defines an isomorphism of bialgebras
\[
\Psi_q : A_{\q}(n) \to (\K_{\q}(n), \cdot, \Delta),
\]
where $(\cdot, \Delta)$ are defined in $(\ref{newmult})$ and $(\ref{newcomult})$.
\end{thm}

\subsection{Automorphisms}

The bstection $\F_{\mbf c}\times \F_{\mbf d}\to \F_{\mbf d} \times \F_{\mbf c} $ defined by $(V, F) \mapsto (F, V)$ induces (say by the pullback function $f^*$ in Section ~\ref{operations}) an automorphism 
\[
\tau_1: A_{\q}(n) \to A_{\q}(n), \quad  E_{st}\mapsto E_{ji} \quad \forall 1\leq i, j\leq n
\] 
as an algebra, and an anti-automorphism of $A_{\q}(n)$ as a coalgebra. 

Given $\mbf c =(c_1, \cdots, c_n)$, we set $\mbf c^t= (c_n, \cdots, c_1)$.  Let $J$ be the matrix 
\[
J=
\begin{pmatrix}
0 & 0 & \cdots 0 & 1 \\
0 & 0 & \cdots 1 & 0 \\
\cdot & \cdot & \cdots  &\cdot\\
1 & 0 & \cdots 0 & 0.
\end{pmatrix}.
\]
Let $P_{\mbf c}$ be a parabolic subgroup of $\GL(d)$ of type $\mbf c$. For example, $P_{\mbf c}$ can be taken as the set of all invertible block matrices that are upper triangular and the block size on the diagonal is $c_1$, $c_2$, $\cdots $, $c_n$. Then $P_{\mbf c^t}= JP_{\mbf c} J$ is a parabolic subgroup of type $\mbf c^t$ consisting of all lower triangular matrices of block size on the diagonal equal to $c_n$, $c_{n-1}$, $\cdots$, $c_1$. If we fix a flag $F_{\mbf c}$ (resp. $F_{\mbf c^t}$) in $\F_{\mbf c}$ (resp. $\F_{\mbf c^t}$), such that the stabilizer of $F_{\mbf c}$ (resp. $F_{\mbf c^t}$) is $P_{\mbf c}$ (resp. $P_{\mbf c^t}$),  then the assignment $gF_{\mbf c} \mapsto gJF_{\mbf c^t}$ defines a bstection $\F_{\mbf c} \to \F_{\mbf c^t}$. 
By a similar manner, we can define a bstection $\F_{\mbf c}\times \F_{\mbf d} \to \F_{\mbf c^t} \times \F_{\mbf d^t}$ such that it induces a bstection on the $\GL(d)$ orbits  given by
\[
\Theta_d(\mbf{c, d}) \to \Theta_d(\mbf c^t, \mbf d^t), \quad M\mapsto JMJ. 
\]
This bstection defines an anti-automorphism of algebras and automorphism of coalgebras
\[
\tau_2: A_{\q}(n) \to A_{\q}(n), \quad E_{st} \mapsto E_{n+1-i, n+1-j}, \quad \forall 1\leq i, j\leq n.
\]

The composition $\tau_3=\tau_1\tau_2$ is then an anti-automorphism of $A_{\q}(n)$ as algebras and as coalgebras. 

Note that $\tau_i$ for $i=1, 2, 3$ are the (anti)-automorphisms defined in ~\cite{PW91}.

\subsection{Proof of Proposition ~\ref{Delta-1}} 
\label{comparison} 
We shall compare the two algebras $(\K_{\q}(n), \circ, \tilde \Delta)$ and $(\K_{\q}(n), \cdot, \Delta)$.
For a triple $(M, M'', M' )\in \Theta_d\times \Theta_{d''}\times \Theta_{d'}$, we fix a pair $(V, F)\in O_M$ and set
\[
g^M_{M'', M'} = \# \{ E\subset \mbb F_q^d | (V, F) \cap E\in O_{M''} ,  (V, F) \cap \frac{\mbb F_q^d}{E} \in O_{M'}\}. 
\]
The definition of $g^M_{M'', M'}$ is independent of the choice of $(V, F)$.  The multiplication ``$\circ$''  in (\ref{twist}) can be rewritten as 
\[
1_{M''}\circ 1_{M'}  = (q^{-\frac{1}{2}})^{ f_1-f_2} \sum_{M\in \Theta} g^M_{M'', M'} 1_M.
\]
We fix element $(V'', F'') \in O_{M''}$ and $(V', F') \in O_{M'}$ and a decomposition $\mbb F_q^{d'' } \oplus \mbb F_q^{d'} = \mbb F_q^d$. 
Let $\mathcal U$ be the set of all flags $\tilde V$ such that $\tilde V\cap \mbb F_q^{d''} = V''$ and $\tilde V\cap \frac{\mbb F_q^d}{\mbb F_q^{d''}} = V'$. 
 Then by an analysis of the diagram (\ref{induction}), we have 
 \[
 g^M_{M'', M'} = \frac{\# \mathcal U}{\#R} \frac{a_M}{a_{M''} a_{M'}  } h^M_{M'', M'},
 \]
 where  $R$ is the unipotent  subgroup in $\GL(d)$ consisting of all elements $g$ 
 such that $g(x) =x $ for $x\in  \mbb F_q^{d''}$ and $g(y) = y $ modulo $\mbb F_q^{d''}$ for any $y\in  \mbb F_q^{d'}$. 
 We have 
 \[
  \frac{\# \mathcal U}{\#R} = q^{\sum_{i<j} c_i'c_j'' -d'd'' } = (q^{-\frac{1}{2}})^{-2\sum_{i<j} c'_ic''_j +2 d'd''},
 \]
 where the notations are in consistent with that in Section ~\ref{mult}.
Thus the multiplication ``$\circ$'' in (\ref{twist}) can be rewritten as 
\begin{equation}
\label{mult-h}
1_{M''}\circ 1_{M'}  = (q^{-\frac{1}{2}})^{ \sum_{i<j} -c_i'c_j'' + d_i'd_j'' + 3d'd''} \sum_{M\in \Theta}   \frac{a_M}{a_{M''} a_{M'}  } h^M_{M'', M'} 1_M.
\end{equation}

By (\ref{mult-h}) and Proposition ~\ref{newalgebrahomo},  we have  proved Proposition ~\ref{Delta-1}.

\subsection{Comparison of $\GL_{\q}(n)$ and $\GL^{DD}_q(n)$}
We define a linear map 
\begin{equation}
\label{DD-mult}
\bullet: \K_{\q}(n)\otimes \K_{\q} (n) \to \K_{\q}(n),
\end{equation}
by 
\[
1_{M''} \bullet 1_{M'} =  \sum_{M\in \Theta}   h^M_{M'', M'} 1_M, \quad \forall M'', M'\in \Theta,
\]
From the argument of the fact that $(\K_{\q}(n), \cdot)$ is an associative algebra, we see that the pair $(\K_{\q}(n),\bullet)$ is an associative algebra. Moreover, 

\begin{prop}
\label{DD-isom}
The assignment of sending    $c_{st}$  to
$E_{ji}$, for any  $1\leq i, j\leq n$, defines an isomorphism of bialgebras
\[
\Psi_q': B_q(n) \to (\K_{\q}(n), \bullet, \Delta),
\]
where $\Delta $ is defined in (\ref{newcomult}).
\end{prop}

\begin{proof}
The fact that the elements $E_{ji}$ satisfy the defining relations for $c_{st}$ in Section ~\ref{qmatrixDD} follows from the computations in the Proof of Proposition ~\ref{newdefining}. 
The fact that the set $\{ E_{ji} | 1\leq i, j\leq n\}$ generates the algebra $(\K_{\q}(n),\bullet)$ follows from Proposition ~\ref{newbasis}, since the multiplications $\cdot$ and $\bullet$ differ only by a twist of a cocycle by comparing the definitions. (In particular, we have $E^{M} =q^{\sum_{i> k, j<l} m_{st} m_{kl}} 1_{M} $, where the multiplication is taken with respect to $\bullet$.)   The fact that $\Delta$ is an algebra homomorphism with respect to the algebra multiplication $\bullet$ follows from the proof of Proposition ~\ref{newalgebrahomo} (by forgetting  the twist).
\end{proof}

 By comparing Theorem ~\ref{FRT-isom}  with Proposition ~\ref{DD-isom}, we have
 
 \begin{cor}
 \label{A-B-coalgebra}
The assignment $E_{st} \mapsto c_{ji}$, for any $1\leq i, j \leq n$ defines an isomorphism of coalgebras
\[
(A_{\q}(n), \Delta)  \to (B_q(n),\Delta).
\]
 \end{cor}

By Theorem ~\ref{FRT-isom} and Proposition ~\ref{DD-isom}, we  can  identify $A_{\q}(n)$ and $B_q(n)$ with $\K_{\q}(n)$ via $\Psi_q$ and $\Psi_q'$.   
Under such identifications, we  regard the quantum determinants $\mrm{det}_{\q}$ in (\ref{FRT-determinant}) and $\mrm{det}_q^{DD}$ in (\ref{DD-determinant}) as elements in $\K_{\q}(n)$.   
We shall identify elements in $S_n$ with the permutation matrices in $\Theta$. 
Under such an identification, we have 
\[
l(\sigma) = \sum_{i>k, j<l} \sigma_{st} \sigma_{kl}, \quad \forall \sigma\in S_n.
\]

\begin{cor} 
\label{determinant}
In $\K_{\q}(n)$, we have 
\[
\mrm{det}_{\q} = \mrm{det}_q^{DD}=\sum_{\sigma\in S_n} (-1)^{l(\sigma)} 1_{\sigma},
\]
which is independent of the choice of the finite field $\mbb F_q$.
\end{cor}

\begin{proof}
By (\ref{FRT-determinant}) and Proposition ~\ref{newbasis}, we have 
\begin{align*}
\mrm{det}_{\q} &= \sum_{\sigma\in S_n} (-\q)^{-l(\sigma)} E_{1, \sigma(1)}\cdots E_{n,\sigma(n)}
=  \sum_{\sigma\in S_n} (-\q)^{-l(\sigma)}  (\q)^{\sum_{i>k, j<l} \sigma_{st} \sigma_{il}} 1_{\sigma} \\
&=\sum_{\sigma \in S_n} (-1)^{l(\sigma)} 1_{\sigma}.
\end{align*}
By (\ref{DD-determinant}) and an analog (or proof)  of Proposition ~\ref{newbasis}, we have
 \begin{align*}
\mrm{det}^{DD}_q &= \sum_{\sigma\in S_n} (-q)^{-l(\sigma)}  c_{\sigma(1), 1} \cdots c_{\sigma(n), n} 
= \sum_{\sigma\in S_n} (-q)^{-l(\sigma)} E_{1,\sigma(1)}\bullet \cdots \bullet E_{n,\sigma(n)}\\
&=  \sum_{\sigma\in S_n} (-q)^{-l(\sigma)}  q^{\sum_{i>k, j<l} \sigma_{st} \sigma_{il}} 1_{\sigma} 
=\sum_{\sigma \in S_n} (-1)^{l(\sigma)} 1_{\sigma}.
\end{align*}
Corollary follows from the above computations.
\end{proof}

By comparing Theorem ~\ref{FRT-isom} and Proposition ~\ref{DD-isom}, we see that the algebra $A_{\q}(n)$ of Faddeev-Reshetikhin-Takhtajan and the algebra $B_q(n)$ of Dipper-Donkin differ only by a twist on the multiplications, though they are not isomorphic in general. More precisely, we define a new multiplication on $B_q(n)$ by 
\[
c_{st} \tilde \bullet  c_{kl} =  q^{\frac{1}{2} \sum_{\alpha<\beta} (e_l)_{\alpha} (e_j)_{\beta} -(e_k)_{\alpha} (e_i)_{\beta}} c_{st} c_{kl}, \quad \forall 1\leq i, j, k, l\leq n,
\]
where $(e_i)_{\alpha}$ denotes the $\alpha$th component of $e_i$.  It is clear that the operation $\tilde \bullet$ is associative. Moreover, we have 

\begin{cor}
\label{A-B}
 The assignment $c_{st} \mapsto E_{ji}$, for any $1\leq i, j\leq n$,  defines an isomorphism from 
the bialgebra $(B_q(n), \tilde \bullet, \Delta) $ to $A_{\q}(n)$.
 \end{cor}

 Another way to state Corollary ~\ref{A-B} is to get rid of the twist in $A_{\q}(n)$ to obtain $B_q(n)$ as follows.
 We define a new multiplication on $A_{\q}(n)$ by 
 \begin{align}
 \label{A-modified}
 E_{st}\;  \tilde \cdot \;  E_{kl} = (\q)^{\sum_{\alpha<\beta} - (e_k)_{\alpha} ( e_j)_{\beta} + (e_l)_{\alpha} (e_j)_{\beta}} 
 E_{st} \cdot  E_{kl}, \quad \forall 1\leq i, j\leq n.
 \end{align}
 Then we have an isomorphism of bialgebras 
 \begin{align}
 \label{B-A}
 B_q(n) \to (A_{\q}(n), \tilde \cdot, \Delta)
 \end{align}
 given by sending $c_{st}$ to $E_{ji}$ for any $1\leq i, j \leq n$.
 
 Since the quantum determinant $ \mrm{det}_{\q}=\mrm{det}^{DD}_{q} $ is central in $A_{\q}(n)$, we see, say ~\cite{Lam},  that
 $\GL_{\q}(n) $ is a Hopf algebra  generated by $E_{st}$ for any $1\leq i, j\leq n$ and 
 the symbol $\mrm{det}_{\q}^{-1} = \mrm{det}^{DD, -1}_q$ subject to the defining relations of $A_{\q}(n)$ and the following
 relations:
 \begin{align}
 \label{A-extra}
 E_{st}  \mrm{det}_{\q}^{-1} =  \mrm{det}_{\q}^{-1} E_{st},\quad 
 \mrm{det}_{\q} \mrm{det}_{\q}^{-1} = 1 = \mrm{det}_{\q}^{-1} \mrm{det}_{\q}, \quad \forall 1\leq i, j\leq n.
 \end{align}
 If we modify the multiplication on $\GL_{\q}(n)$ by (\ref{A-modified}) and the following
 \begin{align*}
 q^{j-i}  E_{st} \; \tilde \cdot  \;  \mrm{det}_{\q}^{-1} &=  \mrm{det}_{\q}^{-1} \;   \tilde \cdot \;  E_{st}
 & & \forall  1\leq i \leq j \leq n,\\
E_{st} \;  \tilde \cdot \;  \mrm{det}_{\q}^{-1}  &= q^{i-j}  \mrm{det}_{\q}^{-1}  \; \tilde \cdot \; E_{st} 
&& \forall 1\leq j< i\leq n,\\
  \mrm{det}_{\q} \; \tilde \cdot \;  \mrm{det}_{\q}^{-1} = 1 &= \mrm{det}_{\q}^{-1} \;  \tilde \cdot  \; \mrm{det}_{\q}, 
  && \forall 1\leq i, j\leq n.
 \end{align*}
 We see that $(\GL_{\q}(n), \tilde \cdot , \Delta)$ is still a  bialgebra.

 From Theorem ~\ref{FRT-isom}, Proposition ~\ref{DD-isom}, Corollary ~\ref{A-B}  and  Corollary ~\ref{determinant}, 
 we have the following  natural generalization of Corollary ~\ref{A-B-coalgebra} and (\ref{B-A}).
 
 \begin{thm}
 \label{FRT-DD}
 The assignment  of sending $c_{st}$  to $E_{ji} $, for any $1\leq i, j\leq n$,  
 and $ \mrm{det}_{q}^{DD,-1}$ to 
 $\mrm{det}_{\q}^{-1}$   defines an isomorphism of bialgebras
 \[
 \Xi: \GL_q^{DD}(n) \to (\GL_{\q}(n), \tilde \cdot, \Delta) .
 \]
 In particular, the coalgebras $(\GL_q^{DD}(n), \Delta)$ and $(\GL_{\q}(n), \Delta)$ are isomorphic via $\Xi$. 
 \end{thm}

 It is clear, from the definitions, that 
\begin{align*}
S \Xi (c_{st}) &= S(E_{ji}) = (-\q)^{i-j} A(i, j) \cdot  \mrm{det}_{\q}^{-1},\\
\Xi S^{DD} (c_{st})& = \Xi ( (-1)^{i+j} A^{DD}(j,i) \mrm{det}^{DD, -1}_q) = (-1)^{i+j} A(i, j) \;\tilde \cdot \;\mrm{det}_{\q}^{-1} =(-\q)^{i-j} A(i, j) \cdot  \mrm{det}_{\q}^{-1},
\end{align*}
where $S$ and $S^{DD}$ are antipodes for $\GL_{\q}(n)$ and $\GL_q^{DD}(n)$ defined in (\ref{antipode}) and (\ref{antipode-DD}), respectively.
So we have 
\[
S \Xi (c_{st})  = \Xi S^{DD} (c_{st}) \quad \forall 1\leq i, j\leq n.
\]
In the isomorphism  $\Xi$, the multiplication of $\GL_{\q}(n)$ has to be modified by a twist of a cocycle. So we shall not expect  the antipode $S$ defined  by using the original multiplication ``$\cdot$'' 
to be compatible with $S^{DD}$ with respect to $\Xi$.  Instead, we let 
\[
\tilde S: \GL_{\q}(n) \to \GL_{\q}(n)
\] 
be the new antipode on $\GL_{\q}(n)$ defined by the same formula (\ref{antipode}) under the new multiplication ``$\;\tilde \cdot\;$''.

\begin{cor}
$\Xi: \GL_q^{DD}(n) \to (\GL_{\q}(n), \tilde \cdot, \Delta, \tilde S)$ is an isomorphism of Hopf algebras. 
\end{cor}

 \begin{rem}
\label{remark-1}
(1).  The quantized coordinate algebra of the space of matrices of size $m\times n$ in ~\cite{NYM93} can be realized geometrically in the same way as that of $A_{\q}(n)$ and $B_q(n)$ by considering the set of double flags of $m$-step  in the first component and $n$-step in the second component.

(2).  All similar results can be obtained over the ring $\mbb Z[\q, q^{-\frac{1}{2}}] $ of Laurent polynomials.

(3). From ~\cite{Lu99} and ~\cite{Lu00}, we see that there are polynomials $c^L_{M, N}(v)$ and $h^{L}_{M, N}(v)$ in   $\mbb Z[v, v^{-1}]$ such that 
$c^L_{M, N}(\q) = c^L_{M, N}$ and $h^L_{M. N}(\q) = h^L_{M, N}$, for any prime power $q$. 
Therefore, the results in this paper can be extended to results on objects over $\mbb Z[v, v^{-1}]$ for $v$ an indeterminate.

(4). Proposition ~\ref{DD-isom} is dual to the geometric construction of $q$-Schur algebras in   ~\cite{BLM90} and ~\cite{Lu00}.

(5).  The expression $\sum_{\sigma\in S_n} (-1)^{l(\sigma)} 1_{\sigma}$ in Corollary ~\ref{determinant} is exactly the Young symmetrizer of $S_n$ which produces the sign representation of $S_n$. The determinant function also defines a representation of $S_n$. So everything matches up. See ~\cite[Lecture 4]{FH91}. We thank Dave Hemmer for pointing out this to us. 

\end{rem}

\subsection{Category $\A$}

Let $\A\equiv \A_{\q}(n)$ be the category as follows: 
\begin{itemize}
\item objects:  $(V, F)\in \F(D) \times \F(D)$,  for any finite dimensional vector space $D$ over $\mbb F_q$;
\item morphisms: $\Hom_{\A} ((V, F), (V', F')) =\{ f: D\to D' \; \mbox{linear maps}\;| f(V, F) \subseteq (V', F')\}$, for any $(V, F) \in \F(D) \times \F(D)$ and $(V', F')\in \F(D')\times \F(D')$. 
\end{itemize} 
We have 

\begin{prop}
$\A$ is an exact   category.
\end{prop}

This can be proved straightforwardly. 

We remark that  $\A$ is not abelian. Take $f=id: D\to D$ and take $(V, F) $ and $(V', F')$ 
such that $V\subsetneq V'$ and $F\subsetneq F'$. Then there is  no cokernel of $f$ in $\A$. 
The algebra $(A_{\q}(n), \circ)$ can then be interpreted as the Hall algebra of $\A$.

\section{Acknowledgment} 

This paper is grown out of the quest to answer the question of defining an Heisenberg algebra action on the function space of double $n$-step flag varieties in analogue with the one in ~\cite{W01}. This question is  raised by Weiqiang Wang during the author's visit at University of Virginia in 2010.  We thank Weiqiang for his kind invitation. 
We also thank Professor Masatoshi Noumi for sending his paper ~\cite{NYM93} to the author and H.C. Zhang and J. Du for helpful comments.
We are very grateful to the referee for the suggestions and corrections in improving this
article.
Partial support from NSF grant  DMS 1160351 is  acknowledged.

\end{document}